\documentclass[a4paper,10pt,oneside]{amsart}

\usepackage{amsmath,amsfonts,amssymb}
\usepackage[all]{xy}
\usepackage[english]{babel}
\usepackage[utf8]{inputenc}
\usepackage{microtype}
\usepackage[hypertexnames=false,
backref=page,
    pdftex,
    pdfpagemode=UseNone,
    breaklinks=true,
    extension=pdf,
    colorlinks=true,
    linkcolor=blue,
    citecolor=blue,
    urlcolor=blue,
]{hyperref}

\reversemarginpar

\newcommand{\referenza}{}

\newtheorem{prop}{Proposition}[section]
\newtheorem{thm}[prop]{Theorem}
\newtheorem*{thm*}{Theorem \referenza}
\newtheorem{cor}[prop]{Corollary}

\theoremstyle{definition}
\newtheorem{rem}[prop]{Remark}

\newcommand{\R}{\mathbb{R}}
\newcommand{\C}{\mathbb{C}}

\newcommand{\sspace}{\cdot}
\newcommand{\ssspace}{\cdot\cdot}

\DeclareMathOperator{\im}{i}
\DeclareMathOperator{\de}{d}

\title{Stability of holomorphically-parallelizable manifolds}

\author{Daniele Angella}

\address[Daniele Angella]{Istituto Nazionale di Alta Matematica\\
Dipartimento di Matematica e Informatica\\
Universit\`{a} di Parma\\
Parco Area delle Scienze 53/A, 43124\\
Parma, Italy}
\email{daniele.angella@gmail.com}

\curraddr{Centro di Ricerca Matematica ``Ennio de Giorgi''\\
Collegio Puteano, Scuola Normale Superiore\\
Piazza dei Cavalieri 3\\
56126 Pisa, Italy
}

\email{daniele.angella@sns.it}

\author{Adriano Tomassini}

\address[Adriano Tomassini]{Dipartimento di Matematica e Informatica\\
Universit\`{a} di Parma \\
Parco Area delle Scienze 53/A, 43124 \\
Parma, Italy}
\email{adriano.tomassini@unipr.it}

\thanks{The first author is granted with a research fellowship by Istituto Nazionale di Alta Matematica INdAM, and is supported by the Project PRIN ``Varietà reali e complesse: geometria, topologia e analisi armonica'', by the Project FIRB ``Geometria Differenziale e Teoria Geometrica delle Funzioni'', by SNS GR14 grant ``Geometry of non-Kähler manifolds'', and by GNSAGA of INdAM. The second author is supported by the Project PRIN ``Varietà reali e complesse: geometria, topologia e analisi armonica'', by Project FIRB ``Geometria Differenziale Complessa e Dinamica Olomorfa'', and by GNSAGA of INdAM.\\[5pt]
To appear in {\em C. R. Math. Acad. Sci. Paris}}

\begin{document}

\begin{abstract}
 We prove a stability theorem for families of holomorphically-parallelizable manifolds in the category of Hermitian manifolds.
\end{abstract}

\maketitle

\section*{Introduction}

By a classical theorem by K. Kodaira and D.~C. Spencer, \cite[Theorem 15]{kodaira-spencer-3}, small deformations of compact complex manifolds admitting a K\"ahler metric still admit K\"ahler metrics. This is a consequence of the harmonicity property of K\"ahler metrics and of Hodge theory on compact K\"ahler manifolds. On the other side, H. Hironaka provided in \cite{hironaka} an example of a complex-analytic family of compact complex manifolds being K\"ahler except as for the central fibre, which is only Mo\v\i\v{s}hezon. In other words, K\"ahlerness is not a closed property under deformations. It is expected that limits of projective manifolds are Mo\v\i\v{s}hezon, and limits of K\"ahler manifolds are in class $\mathcal{C}$ of Fujiki, see \cite{demailly-paun, popovici-invent}. 
Note that $\partial\overline\partial$-Lemma is an invariant property under images of holomorphic birational maps, \cite[Theorem 5.22]{deligne-griffiths-morgan-sullivan}. In particular, compact complex manifolds being Mo\v\i\v{s}hezon or belonging to class $\mathcal{C}$ of Fujiki satisfy the $\partial\overline\partial$-Lemma, \cite[Corollary 5.23]{deligne-griffiths-morgan-sullivan}.
Hence, in view of the above conjectures, in \cite{angella-kasuya-2, angella-tomassini-3}, the behaviour of the $\partial\overline\partial$-Lemma property under deformation is investigated. In particular, \cite[Corollary 2.7]{angella-tomassini-3} provides another argument for proving the stability of $\partial\overline\partial$-Lemma under small deformations; see the references therein for different proofs. While, in \cite[\S4, Corollary 6.1]{angella-kasuya-2}, it is provided an example showing that $\partial\overline\partial$-Lemma is not stable under limits. Note in fact that the structures on the holomorphically-parallelizable Nakamura manifold studied in \cite[\S4]{angella-kasuya-2} are not in class $\mathcal{C}$ of Fujiki. In any case, nilmanifolds and solvmanifolds (that is, compact quotients of connected simply-connected nilpotent, respectively solvable, Lie groups by co-compact discrete subgroups,) provide a possibly useful class of examples for investigating the above questions. In fact, K\"ahlerness 
for nilmanifolds is characterized by $\partial\overline\partial$-Lemma, see \cite[Theorem 1, Corollary]{hasegawa}, which is in turn characterized in terms of Bott-Chern cohomology, \cite[Theorem B]{angella-tomassini-3}. On the other hand, several results concerning computation of Bott-Chern cohomology for nilmanifolds and solvmanifolds are known, see, e.g., \cite{angella-kasuya-1, angella-1} and the references therein. When restricting to the class of nilmanifolds, K\"ahlerness is a closed properties under deformations. This follows by a theorem by A. Andreotti, H. Grauert, and W. Stoll in \cite{andreotti-stoll}. More precisely, they proved a stability result for complex-analytic families of complex tori.

In this short note, we prove a similar result than A. Andreotti, H. Grauert, and W. Stoll for the class of holomorphically-parallelizable manifolds, in the category of compact Hermitian manifolds. As pointed out by the Referee, it remains an open question whether the result may be stated in the category of compact complex manifolds.

\bigskip

\noindent{\sl Acknowledgments.} The authors would like to thank Paul Gauduchon and the anonymous Referees for their valuable suggestions.

\section{Main results}

A compact complex manifold is called {\em holomorphically-parallelizable} if its holomorphic tangent bundle is holomorphically-trivial, see \cite[page 771]{wang}. A structure theorem for holomorphically-parallelizable manifolds was proven by H.-C. Wang. More precisely, holomorphically-parallelizable manifolds have a complex Lie group as universal covering.

\begin{thm}[{\cite[Theorem 1]{wang}}]
 Let $X$ be a holomorphically-parallelizable manifold. Then $X$ is (biholomorphic to) a quotient $\left. G \middle\slash D \right.$ where $G$ is a connected simply-connected complex Lie group and $D$ is a discrete subgroup.
\end{thm}

Holomorphically-parallelizable manifolds having a complex solvable Lie group as universal covering were studied by I. Nakamura in \cite{nakamura}. He initiated a classification of holomorphically-parallelizable solvmanifolds up to complex dimension $5$ in \cite[\S6]{nakamura}, then completed by D. Guan in \cite{guan}. Moreover, by explicitly constructing the Kuranishi family of deformations of some holomorphically-parallelizable solvmanifolds of complex dimension $3$, in \cite[\S3]{nakamura}, it was proved that being holomorphically-parallelizable is not a stable property under small deformations of the complex structure, \cite[page 86]{nakamura}. A detailed study of holomorphically-parallelizable nilmanifolds, and of their Kuranishi space and stability was done by S. Rollenske in \cite{rollenske-jems}.

\medskip

The structure theorem by H.-C. Wang allows to generalize and simplify a stability result by A. Andreotti, H. Grauert, and W. Stoll, \cite[Theorem 8]{andreotti-stoll}.
In the proof below, the classical and well-known Montel theorem, (also called generalized Vitali theorem,) is used.
\begin{thm}[{Montel theorem; see, e.g., \cite[Proposition 6]{narasimhan}}]
 Let $\mathcal{F}=\{f\}$ be a family of holomorphic functions on an open set $\Omega\subseteq\C^n$ such that, for any compact set $K\subseteq\Omega$, there exists a positive constant $M_K$ such that, for any $z\in K$, for any $f\in\mathcal{F}$, it holds $|f(z)|<M_K$. Then any sequence $\{f_n\}_n\subseteq\mathcal{F}$ contains a subsequence which converges uniformly on compact subsets of $\Omega$.
\end{thm}

We can now state and prove the main result of this note.

\begin{thm}\label{thm:main}
 Let $\left\{ \left( X_t , g_t \right) \right\}_{t \in (-\varepsilon,1)}$ be a smooth family of compact Hermitian manifolds, with $\varepsilon>0$ small enough. Suppose that $X_t$ is holomorphically-parallelizable for any $t\in(0,1)$, with a pointwise $g_t$-orthonormal co-frame $\left\{ \varphi^j(t) \right\}_j$ of holomorphic $1$-forms.
 Then $X_0$ is holomorphically-parallelizable.
\end{thm}

\begin{proof}
 First of all, by the Ehresmann theorem, for any $t\in(-\varepsilon, 1)$, we see $X_t=(X,J_t)$ where $\{J_t\}_{t\in(-\varepsilon,1)}$ is a family of complex structures on the differentiable manifold $X$ varying smoothly in $t$.
 
 By definition, the holomorphic tangent bundle $T^{1,0}X_t$ of $X_t$ is holomorphically-trivial for any $t\in (0,1)$. Equivalently, the holomorphic co-tangent bundle $\left(T^{1,0}X_t\right)^*$ of $X_t$ is holomorphically-trivial for any $t\in (0,1)$.
 Hence, by the assumptions, we choose $\left\{\varphi^1(t),\ldots, \varphi^n(t)\right\}$ global pointwise $g_t$-orthonormal co-frame of holomorphic $1$-forms on $X_t$ depending smoothly on $t$, where $n$ denotes the complex dimension of $X_t$.
 
 Denote by $\left(\sspace, \ssspace\right)_t$ the induced $\mathrm{L}^2$-Hermitian product on $1$-forms, defined as $\left( \varphi, \psi \right)_t:=\int_X \varphi\wedge*_{g_t}\bar\psi$, where $*_{g_t}$ denotes the Hodge-$*$-operator associated to $g_t$.

 For any fixed $z_0 \in X_0$, consider a local holomorphic coordinate chart
 $$ \left( U \times (-\delta,\delta), \, \left( z^1=x^1+\im x^2, \ldots, z^n=x^{2n-1}+\im x^{2n}, t \right) \right) $$
 centered at $(z_0,0)$ on $\left\{X_t\right\}_{t\in(-\varepsilon,1)}$. Locally on $U \times (0,\delta)$, for $j\in\{1,\ldots,n\}$,
 $$ \varphi^{j}(z,t) \;\stackrel{\text{loc}}{=}\; \sum_{\alpha=1}^{n} \varphi^j_{\alpha}(z^1,\ldots,z^n,t)\, \de z^\alpha \qquad \text{ in } U \times (0,\delta) \;, $$
 where $\left\{ \varphi^j_{\alpha}(z^1,\ldots, z^n,t) \right\}_{\alpha}$ are smooth in $(z^1,\ldots, z^n,t)$ and holomorphic in $(z^1,\ldots,z^n)$.
 
 We claim that, for any $\alpha \in \{1,\ldots,n\}$, the set $\left\{ \varphi^j_{\alpha}(z^1,\ldots, z^n,t) \right\}_{t\in\left(0,\frac{\delta}{2}\right]}$ is a uniformly-bounded family of holomorphic functions on compact subsets. More precisely, this follows from the following two observations. Fix an open relatively compact $V$ in $U$ and consider $t$ varying in $\left(0,\frac{\delta}{2}\right]$. First, there exists a positive constant $C$ such that, for every $(z,t)\in \overline{V} \times \left(0,\frac{\delta}{2}\right]$, it holds
 $$ \sum_{\alpha,\beta=1}^{n} g_t^{\alpha,\bar\beta}(z) u_\alpha \bar u_\beta \;\geq\; C\, \sum_{\gamma=1}^{n} \left|u_\gamma\right|^2 \qquad \text{ and } \qquad \sqrt{\det\left(g_{t,\; \alpha,\beta}(z)\right)_{\alpha,\beta}} \;\geq\; C $$
 where $\left(u_\alpha\right)_{\alpha\in\{1,\ldots,n\}}$ is a vector in $\C^n$.
 Second, for any $j\in\{1,\ldots,n\}$, the $1$-form $\varphi^{j}(t)$ has uniformly-bounded norm with respect to the $\mathrm{L}^2$-Hermitian product induced by $g_t$. Therefore we have
 \begin{eqnarray*}
  \lefteqn{ \max_{t\in\left[0,\frac{\delta}{2}\right]} \mathrm{Vol}(X,g_t) } \\[5pt]
  &\geq& \left( \varphi^j(t) , \varphi^j(t) \right)_t \\[5pt]
  &=& \int_{X} \sum_{\alpha,\beta} g_t^{\alpha\bar\beta}(z) \varphi^j_{\alpha}(z,t) \bar\varphi^j_{\beta}(z,t) \, \sqrt{\det\left(g_{t,\; \alpha,\beta}(z)\right)_{\alpha,\beta}} \, \de x^1 \wedge\cdots\wedge \de x^{2n} \\[5pt]
  &\geq& \tilde C^2 \, \int_{V} \sum_{\gamma=1}^{n} \left| \varphi^j_{\gamma}(z,t) \right|^2 \, \de x^1 \wedge\cdots\wedge \de x^{2n}
 \end{eqnarray*}
 from which we get
 $$ C^{-2} \;\geq\; \int_V \left| \varphi^j_{\gamma}(z,t) \right|^2 \, \de x^1 \wedge\cdots\wedge \de x^{2n} \;. $$
 
 Classical estimates now imply that, given a compact subset $K$ in $V$, there exists a constant $A_K>0$ such that $\left|\varphi^j_\gamma(z,t)\right|\leq A_k$ for any $(z,t)\in K\times \left(0,\frac{\delta}{2}\right]$.
 Then we apply the Montel theorem. In fact, for fixed $j$ and $\gamma$, the family $\left\{\varphi^j_{\gamma}(z^1,\ldots,z^n,t)\right\}_{t\in\left(0, \frac{\delta}{2}\right]}$ is a family of holomorphic functions on $V$ uniformly bounded on compact subsets $K$ in $V$. Then, up to pick-out a subsequence, $\varphi^j_\alpha(z^1,\ldots,z^n,t)$ converges uniformly on compact subsets, as $t\to 0$, to a holomorphic function $\varphi^j_\alpha(z^1,\ldots,z^n,0)$ on $V$.

 Now, we take another local holomorphic coordinate chart intersecting the first one. Then we extract another subsequence in such a way that the convergence holds for both charts. By continuing in this way, we can construct a set $\left\{ \varphi^{1}(0),\ldots,\varphi^{n}(0) \right\}$ of holomorphic $1$-forms on $X_0$.
 
 We claim that this set is orthogonal with respect to the $\mathrm{L}^2$-Hermitian product associated to $g_0$. Indeed, consider a smooth partition of unity $\{\rho_\ell\}_{\ell}$ associated to a covering $\{U_\ell\}_{\ell}$ with coordinate holomorphic charts as in notations above. We have
 \begin{eqnarray*}
  \lefteqn{\left( \varphi^j(0), \varphi^k(0) \right)_{0}} \\[5pt]
  &=& \sum_\ell \int_{U_\ell} \rho_\ell\, \sum_{\alpha,\beta=1}^n g^{\alpha\bar\beta}_{\ell,\,0}(z) \varphi^j_{\ell,\, \alpha}(z,0)\overline{\varphi^k_{\ell,\, \beta}(z,0)} \, \sqrt{\det\left(g_{t,\; \alpha,\beta}(z)\right)_{\alpha,\beta}} \, \de x^1 \wedge\cdots\wedge \de x^{2n} \\[5pt]
  &=& \sum_\ell \int_{U_\ell} \rho_\ell\, \sum_{\alpha,\beta=1}^n \lim_{t\to 0} g^{\alpha\bar\beta}_{\ell,\,t} (z) \varphi^j_{\ell,\,\alpha}(z,t)\overline{\varphi^k_{\ell\,\beta}(z,t)} \, \sqrt{\det\left(g_{t,\; \alpha,\beta}(z)\right)_{\alpha,\beta}} \, \de x^1 \wedge\cdots\wedge \de x^{2n} \\[5pt]
  &=& \lim_{t\to 0} \sum_\ell \int_{U_\ell} \rho_\ell\, \sum_{\alpha,\beta=1}^n g^{\alpha\bar\beta}_{\ell,\,t} (z) \varphi^j_{\ell,\,\alpha}(z,t)\overline{\varphi^k_{\ell\,\beta}(z,t)} \, \sqrt{\det\left(g_{t,\; \alpha,\beta}(z)\right)_{\alpha,\beta}} \, \de x^1 \wedge\cdots\wedge \de x^{2n} \\[5pt]
  &=& \lim_{t\to 0} \left( \varphi^j(t), \varphi^k(t) \right)_t \;=\; \lim_{t\to0} \delta^{jk} \;=\; \delta^{jk} \;.
 \end{eqnarray*}
 
 Now, we claim that $\left\{\varphi^j(0)\right\}_{j\in\{1,\ldots,n\}}$ are in fact linearly independent at every point.
 Indeed, consider a coordinate holomorphic chart as in notations above. Then 
 $$\varphi^{1}(z,t)\wedge\cdots\wedge\varphi^{n}(z,t) = \det\left(\varphi^{j}_{\alpha}(z,t)\right)_{j,\alpha} \, \de z^1 \wedge\cdots\wedge\de z^n\,.$$
 The holomorphic functions $\det\left(\varphi^{j}_{\alpha}(z,t)\right)_{j,\alpha}$ converge to $\det\left(\varphi^{j}_{\alpha}(z,0)\right)_{j,\alpha}$ for $t\to0$ uniformly on compact subsets. We show that $\det\left(\varphi^{j}_{\alpha}(z,0)\right)_{j,\alpha}$ is not identically zero.
 Indeed, $\{\varphi^1\lfloor_p(t),\ldots,\varphi^n\lfloor_p(t)\}$ being orthonormal with respect to $g_t\lfloor_p$ at any point $p\in X$, we have:
 \begin{eqnarray*}
  \lefteqn{ \int_X \left| \det\left(\varphi^{j}_{\alpha}(0)\right)_{j,\alpha}\right|^2 \de z^1\wedge\cdots \de z^n\wedge\de\bar z^1\wedge\cdots\de\bar z^n } \\[5pt]
  &=& \lim_{t\to0} \int_X \left| \det\left(\varphi^{j}_{\alpha}(t)\right)_{j,\alpha}\right|^2 \de z^1\wedge\cdots \de z^n\wedge\de\bar z^1\wedge\cdots\de \bar z^n\\[5pt]
  &=& \lim_{t\to0} c(n)\cdot \mathrm{Vol}(X,g_t) \\[5pt]
  &=& c(n) \cdot \mathrm{Vol}(X,g_0) \;>\; 0 \;,
 \end{eqnarray*}
 where $c(n)$ is a constant depending just on the complex dimension. It follows that $\left| \det\left(\varphi^{j}_{\alpha}(z,0)\right)_{j,\alpha}\right|^2$ is not identically zero.
 Therefore, up to pick-out a subsequence, by the Bochner theorem, \cite[Theorem VIII.8]{bochner-martin}, we obtain that  $\det\left(\varphi^{j}_{\alpha}(z,0)\right)_{j,\alpha}$ is nowhere vanishing, proving the claim.
 
 Therefore $\left\{\varphi^j(0)\right\}_{j\in\{1,\ldots,n\}}$ provides a global co-frame of holomorphic $1$-forms for $X_0$, and hence $X_0$ is holomorphically-parallelizable. 
\end{proof}

\medskip

Note that if $\de\varphi^j(t)=0$, then also $\de\varphi^j(0)=0$, for any $j\in\{1,\ldots,n\}$. In particular, by \cite[Theorem 1]{wang}, one recovers the stability result for differentiable families of complex tori by A. Andreotti, H. Grauert, and W. Stoll.

\begin{cor}[{\cite[Theorem 8]{andreotti-stoll} by A. Andreotti, H. Grauert, and W. Stoll}]
 Let $\left\{ X_t \right\}_{t \in (-\varepsilon,1)}$ be a differentiable family of compact complex manifolds, with $\varepsilon>0$ small enough. Suppose that $X_t$ is a complex torus for any $t\in(0,1)$. Then $X_0$ is a complex torus.
\end{cor}

\begin{proof}
 We take a co-frame $\{\varphi^j(t)\}_j$ of holomorphic $1$-forms on $X_t=\left.\Gamma_t\middle\backslash \C^n \right.=(X,J_t)$ varying smoothly in $t$.
 We set $g_t:=\sum_j \varphi^j(t)\odot\bar\varphi^j(t)$. We claim that $g_0:=\lim_{t\to 0}g_t$ is a Hermitian metric on $X_0$.
 Indeed, we have
 \begin{eqnarray*}
 \lefteqn{ \int_X \varphi^1(0)\wedge\cdots\varphi^n(0)\wedge\bar\varphi^1(0)\wedge\cdots\bar\varphi^n(0) } \\[5pt]
 &=& \lim_{t\to0} \int_X \varphi^1(t)\wedge\cdots\varphi^n(t)\wedge\bar\varphi^1(t)\wedge\cdots\bar\varphi^n(t) \\[5pt]
 &=& \lim_{t\to0} \det \left(\begin{array}{c}\Omega(t) \\ \bar\Omega(t) \end{array}\right) \\[5pt]
 &=& \det \left(\begin{array}{c}\Omega(0) \\ \bar\Omega(0) \end{array}\right) \neq 0
 \end{eqnarray*}
 by \cite[page 341]{andreotti-stoll}, where $\Omega(t)$ is the period matrix of $X_t$.
 
 Finally, we claim that a compact complex holomorphically-parallelizable manifold $X = \left. G \middle\slash D \right.$ of complex dimension $n$ is a torus if and only if the first Betti number is $b_1 = 2n$.
 
 Indeed, note that $b_1 = \dim_\R \left. G \middle\slash \left[G,G\right] \right.$ by the Sakane theorem, \cite[Theorem 1]{sakane}, see also \cite[Corollary 1]{wang}.
 The statement follows, since $t \mapsto b_1(X_t)$ is locally constant at $0$ by the Ehresmann theorem.
\end{proof}

\medskip

In particular, one gets the following.

\begin{cor}
  In the class of holomorphically-parallelizable manifolds, respectively nilmanifolds, the property of being K\"ahler is stable for both small and large deformations.
\end{cor}

\begin{proof}
 By \cite[Corollary 2]{wang}, holomorphically-parallelizable manifolds admit K\"ahler metrics if and only if they are complex tori. Respectively, by \cite[Theorem A]{benson-gordon}, nilmanifolds admit K\"ahler structures if and only if they are tori. By \cite[Theorem 8]{andreotti-stoll}, the statement follows.
\end{proof}

\medskip

\begin{rem}
 Complex solvable Lie groups up to complex dimension $5$ are classified by I. Nakamura in \cite[\S6]{nakamura}. For example, the family in class (IV.5) is characterized by the structure equations
 $$ \de \varphi^1 \;=\; 0 \;, \quad \de \varphi^2 \;=\; \varphi^1\wedge\varphi^2 \;, \quad \de\varphi^3 \;=\; \alpha\,\varphi^1\wedge\varphi^3 \;, \qquad \de\varphi^4 \;=\; -(1+\alpha)\, \varphi^1\wedge\varphi^4 \;, $$
 where $\alpha\in\C$ is such that $\alpha(1+\alpha)\neq0$. Note that, for $\alpha\to 0$, the above family degenerates to the class (IV.4), which is characterized by the structure equations
 $$ \de \varphi^1 \;=\; 0 \;, \quad \de \varphi^2 \;=\; 0 \;, \quad \de\varphi^3 \;=\; \varphi^2\wedge\varphi^3 \;, \qquad \de\varphi^4 \;=\; \varphi^2\wedge\varphi^4 \;. $$
 
 Holomorphically-parallelizable solvmanifolds up to complex dimension $5$ are classified by D. Guan in \cite[Classification Theorem, Theorem 2, Theorem 3]{guan}. In this case, we cannot find an example showing that the property of being holomorphically-parallelizable with a fixed universal covering is not preserved at the limit. We wonder whether such an example can be found.
\end{rem}

\begin{rem}
 As pointed out by the Referee, we ask whether the result in Theorem \ref{thm:main} may be stated in the category of compact complex manifolds.
\end{rem}

\end{document}